\documentclass[a4paper,12pt]{amsart}
\usepackage{amsmath}
\usepackage[bookmarksnumbered, colorlinks, plainpages]{hyperref}
\usepackage{amsfonts}
\usepackage{setspace}
\usepackage{amsthm}
\usepackage{amssymb}
\usepackage{graphicx}
\usepackage{stmaryrd}
\usepackage{yfonts}
\usepackage{color}
\usepackage{cite}
\usepackage[all,cmtip]{xy}
\usepackage{amscd,graphicx, color}
\date{}

 \theoremstyle{plain}
\newtheorem{theo}{Theorem}[section] \theoremstyle{plain}
 \theoremstyle{plain}
 \theoremstyle{definition}
 \theoremstyle{definition}
 \theoremstyle{remark}
\newtheorem{rem}[theo]{Remark} \theoremstyle{remark}

\newtheorem{proposition}[theo]{Proposition}

\begin{document}
\setcounter{page}{1}

\title{ Minimal and Maximal Operator Space Structures on Banach Spaces }

\author{Vinod Kumar  P. }
\address{Department of Mathematics,\\ Thunchan Memorial Govt. College, Tirur,\\
Kerala , India. \em E-mail: vinodunical@gmail.com  \em}
   \author{  M. S. Balasubramani}
   \address{Department of Mathematics, University of Calicut,
Calicut University. P. O., Kerala, India. \em E-mail:
 msbalaa@rediffmail.com \em}

\date{}
\maketitle

\begin{abstract}

Given a Banach space $X$, there are many operator space structures
possible on $X$, which all have $X$ as their first matrix level.
Blecher and Paulsen \cite{bp91}  identified two extreme
 operator space structures on $X$, namely $Min(X)$ and $Max(X)$ which represents respectively, the smallest and the largest
operator space structures admissible on $X$.  In this note, we
consider the subspace and the quotient space structure of minimal
and maximal operator spaces.
\flushleft {AMS Mathematics Subject Classification(2000) No: 46L07,47L25}

\end{abstract}
\thispagestyle{empty}
\section{Introduction }
 Operator spaces form a natural quantization of Banach spaces and
their study took a rigorous form with the representation theorem
obtained by Z. J. Ruan \cite{r88} in 1988 and after that it has
seen considerable development with applications to the theory of
operator algebras and various aspects of operator spaces are being
studied extensively.

 A {\em concrete operator space\em} \index{operator space!concrete} $X$ is a closed
 linear subspace of $ {B(\mathcal{H})} $, for some Hilbert space $\mathcal H$.
 Here, in each matrix level $M_n(X) $, we have a norm
$\left\| . \right\|_n$, induced by the inclusion $ M_n(X) \subset
M_n( B(\mathcal{H}))  $, where the norm in $M_n({B(\mathcal{H})})$
is given by the natural identification  $M_n({B(\mathcal{H}))}
\cong  {B(\mathcal{H}}^n)$.\ More precisely, for $[x_{ij}]\in M_n
(X)$, we have
$$\| [x_{ij}]\|_n = \mbox{sup} \{ ( \Sigma_{i=1}^n
\left\|\Sigma_{j=1}^n x_{ij}h_j \right\|_{\mathcal{H}}^2 )^{1/2}\
|\  h_j \in \mathcal{H}, (\Sigma\| h_j\|_\mathcal{H}^2 )^{1/2}\leq
1 \}. $$ Thus, a concrete operator space carries not just an
inherited norm, but these additional sequence of matrix norms.

 An {\em abstract operator space\em},\index{operator space!abstract} or
 simply an {\em operator space\em} is a pair $ (X, \{  \left\| . \right\| _n \}_{n \in \mathbb{N}} )$
 consisting of a linear space $X$ and a complete norm $ \left\| . \right\| _n $ on  $ M_n(X) $ for every $ n\in \mathbb{N} $,
 such that there exists  a linear complete isometry $\varphi : X \rightarrow B(\mathcal H)$ for some Hilbert space
  $\mathcal H $.
 The sequence of matrix norms $ \{  \left\| . \right\| _n \}_{n \in \mathbb{N}} $ is called an
{\em operator space structure \em} on the linear space $X$.
 An operator space structure $ \{  \left\| . \right\| _n \}_{n \in \mathbb{N}} $ on a Banach space $(X, \left\|.\right\|)$
 is said to be an \textit{{{admissible}}} operator space structure on
 $X$, if $   \left\| . \right\|_1 = \left\|.\right\| $.
An important type of operator spaces are those $X \subset
B(\mathcal H)$
 which are isometric (as a Banach space) to a Hilbert space. Such
  spaces are called \emph{{Hilbertian}} operator spaces
  \cite{gp03}.

If $X$ is a Banach space, then any linear isometry from $X$ to
$B(\mathcal{H})$, for some Hilbert space
 $\mathcal{H}$, endows an operator space structure on $X$. Generally, for a Banach space $X$,
 the matrix norms so obtained are not unique. In other words, a given Banach space has, in general, many realizations
as an operator space. Blecher and  Paulsen  observed that the set
of all operator space structures admissible on a given Banach
space $X$ admits a minimal and a maximal element. The minimal and
the maximal operator space structures on  a Banach space were
introduced and their dual relations were explored in \cite{bp91}
and further structural properties were investigated in \cite{vp92}
and \cite{vp96}. In what follows, we focus on the
 subspace and quotient space structure of minimal and maximal
operator spaces.

It is known that any subspace of a minimal operator space is again
minimal, but quotient of a minimal space need not be minimal.
 A subspace of a maximal operator
space need not be maximal.  But quotient spaces inherits the
maximality. We address  the following question: If every proper,
nontrivial subspace of an operator space $X$ is minimal (maximal),
is $X$ minimal? (maximal?). We give an example to show that the
answer is negative in the case of finite dimensional operator
spaces, and  show that the answer is affirmative in the case of
infinite dimensional operator spaces.

Regarding quotient operator spaces, we prove that, if $X$ is an
infinite dimensional operator space, and if  every quotient of
 $X$ by a proper closed nontrivial subspace of $X$ is minimal (maximal), then
$X$ is minimal (maximal). We also give an example to show that the
result is invalid in the case of finite dimensional operator
spaces.
\section{Minimal and Maximal Operator Spaces}
 Let $ X $ be a Banach space and $X^*$ be its dual
space. Let $ K = Ball(X^*)$ be the  closed unit ball of the dual
space of $X$
 with its weak* topology. Then the  canonical embedding $J : X
   \rightarrow C(K) $,  defined by
   $J(x)(f) = f(x), x \in X $ and $ f \in K $ is a linear isometry. Since,
    subspaces of $C^*$-algebras    are operator spaces (by Gelfand-Naimark Theorem),
      this identification of $X$ induces matrix norms
    on $M_n(X)$ that makes $X$ an operator space and
  the matrix norms on $X$ are given by\\ \centerline{$ \left\| [x_{ij}] \right\|_n
  = \mbox{sup}\{  \left\| [ f(x_{ij}) ]  \right\| \ |\   f\in   K\}$}
 for all  $[x_{ij}] \in M_n(X)$ and for all $n\in \mathbb{N}$.

 Here $\left\| [ f(x_{ij}) ]  \right\|$ indicates
the norm of the
 scalar $n \times n$ matrix $ [ f(x_{ij}) ]  $ viewed as a linear
 map from $\mathbb{C}^n \to \mathbb{C}^n$.

The above defined operator space structure on $X$ is called
   the {\emph{minimal operator space structure }}\index{operator space structure!minimal}on $X$, and we denote this
   operator space as $Min(X)$.
Thus, $Min(X)$ can be regarded as a space of continuous functions
defined on the  closed unit ball of  $X^*$.
 An operator space $X$
is  said to be \emph{{minimal}} if $Min(X)= X$.
 The minimal operator space structure of a Banach space is characterized by
   the universal property that for any arbitrary operator space $Y$,
   any bounded linear map $\varphi : Y \rightarrow Min(X)$ is
   completely bounded and satisfies {$\left\| \varphi : Y \rightarrow Min(X)  \right\|_{cb} =
    \left\| \varphi : Y \rightarrow X  \right\|.$}
The above described universal property implies that $Min(X)$ is
indeed the smallest admissible
    operator space structure on a Banach space $X$. For, if $\{ \|.\|^{'}_n
    \}_{n \in \mathbb{N}}$ is any other admissible operator space structure on $X$,
    and if $\widetilde{X}$ denotes the space $X$ with these matrix
    norms, then $id: \widetilde{X} \to Min(X)$ is a linear
    isometry and $\|id\|_{cb}= \|id\|=1$. This shows that
    $\|.\|^{'}_n$ dominates the corresponding matrix norms in $Min(X)$.
It is known that, an operator space is minimal if and only if it
is completely isometric to a subspace of a commutative
$C^*$-algebra \cite{er02}.

If $X$ is a Banach space, there is a \emph{maximal} way to
consider it as an operator space. For $[x_{ij}]\in M_n(X)$, the
matrix norms given by \\ \centerline{$\left\| [x_{ij}]\right\| =
\mbox{sup} \left\| [\varphi(x_{ij})]  \right\| $} where the
supremum is taken over all operator spaces $Y$ and all linear maps
$\varphi \in Ball(B(X, Y))$,  define an admissible operator space
structure on $X$. We denote this operator space as $Max(X)$ and is
called the \emph{{maximal operator space structure}} on $X$. An
operator space $X$ is said to be  \emph{{maximal}} if $Max(X) = X
$. By the definition of $Max(X)$, any operator space structure
that we can put on $X$, must be smaller than $Max(X)$.

The maximal operator space structure of a Banach space is
characterized by the universal property that for any arbitrary
operator space $Y$, any bounded linear map $\varphi : Max(X)
\rightarrow Y$ is completely bounded and satisfies  \\
\centerline{$\left\| \varphi : Max(X) \rightarrow Y  \right\|_{cb}
=
    \left\| \varphi : X \rightarrow Y  \right\|.$} Thus, if $X$ and
    $Y $ are Banach spaces and $\varphi \in B (X, Y)$, then
    $\varphi$ is completely bounded and $\left\| \varphi   \right\|_{cb} =
    \left\| \varphi   \right\|$, when considered  as a map from $Max( X) \rightarrow
    Y.$

To see that the space $Max(X)$ satisfies the above mentioned
property, let  $\varphi : Max(X) \rightarrow Y$ be a bounded
linear map. Then $u=\dfrac{\varphi}{\|\varphi\|} \in Ball(X,Y)$,
and by definition of the matrix norms of $Max(X)$, $\|
[u(x_{ij})]\|$ is dominated by $\|[x_{ij}]\|_{M_n(Max(X))}$, for
all $[x_{ij}] \in M_n(X)$ and for all $n \in \mathbb{N}$, showing
that $\|u\|_{cb}\leq 1$. Thus $\|{\varphi}\|_{cb}\leq
{\|\varphi\|}$. Therefore, $\varphi : Max(X) \rightarrow Y$ is
completely bounded and $\left\| \varphi   \right\|_{cb} =
    \left\| \varphi   \right\|$.

 The above described universal property implies that $Max(X)$ is indeed the largest admissible
    operator space structure on a Banach space $X$. For, if $\{ \|.\|^{'}_n
    \}_{n \in \mathbb{N}}$ is any other admissible operator space structure on $X$,
    and if $\widetilde{X}$ denotes the space $X$ with these matrix
    norms, then $id: Max(X) \to \widetilde{X}$ is a linear
    isometry and $\|id\|_{cb}= \|id\|=1$. This shows that
    $\|.\|^{'}_n$ is dominated by the corresponding matrix norms in $Max(X)$.

 The following proposition gives  characterizations of minimal and maximal operator
spaces up to complete  \emph{{isomorphisms}}. These
characterizations identify  larger classes of operator space
structures which are completely isomorphic (need not be completely
isometric) to minimal and to maximal operator spaces.
\begin{proposition}\label{new}
\item(i). An operator space $X$ is completely isomorphic to a
minimal operator space if and only if for any arbitrary operator
space $Y$,  any completely bounded linear bijection $\varphi : X
\rightarrow Y$ is a  complete isomorphism.
    \item(ii). An operator space $X$ is completely
isomorphic to a maximal operator space if and only if for any
arbitrary operator space $Y$,  any completely bounded linear
bijection $\varphi : Y \rightarrow X$ is  complete isomorphism.
\end{proposition}
\begin{proof}\hfill

We prove only (i) and (ii) will follow in a similar way. Assume
that $\varphi: X \to Y$ is a completely bounded  linear bijection.
Let $\psi: X \rightarrow Min(Z)$ be a complete isomorphism. Then
by the universal property of minimal operator spaces, $\psi \circ
\varphi^{-1} : Y \rightarrow Min(Z)$ is completely bounded.
Therefore,  $\| \varphi^{-1}\|_{cb} = \| \psi^{-1}\circ \psi \circ
\varphi^{-1}\|_{cb}\leq \| \psi^{-1}\|_{cb} \| \psi
\circ\varphi^{-1}\|_{cb} < \infty.$ This shows that $\varphi : X
\rightarrow Y$ is a  complete isomorphism. For the converse, take
$Y = Min(X)$ and consider the formal identity mapping $id: X
\rightarrow Min(X)$. By assumption, $id^{-1}: Min(X) \rightarrow
X$ is completely bounded, showing that $X$ is completely
isomorphic to $Min(X)$.
\end{proof}
\begin{rem}\hfill

The above theorem describes the  complete isomorphism class of
minimal and maximal   operator spaces.  Recently, T. Oikhberg
\cite{oik07} proved that the complete isomorphism class of any
infinite dimensional operator space has infinite diameter with
respect to the completely bounded Banach-Mazur distance. i.e., for
$n \in \mathbb{N}$, $\forall \ C > 0$, and for any infinite
dimensional operator space $X$, there exists an operator space
structure $\tilde{X}$ on $X$ such that the identity map $id: X \to
\tilde{X}$ is a complete isomorphism, $id^{(n)}$ is an isometry,
and $d_{cb}(X, \tilde{X})> C.$ Measuring the diameter of the
complete isomorphism class of an operator space is still open in
the case of finite dimensional operator spaces.
\end{rem}

Let $X$ be an infinite dimensional Banach space. Then  operator
space structures on $X$, which are  completely isomorphic to
$Min(X)$ can be constructed as follows:
 Choose an operator space $Y$ which is isometric to $X$ and
completely isomorphic to $Min(X)$, say $ v: Min(X) \to Y $ be a
complete isomorphism (From the above remark, such a choice is
always possible). On $X$, define a new operator space structure,
say $\widetilde X$, by setting $\|[x_{ij}]\|_{M_n(\widetilde X)}=
\|[v(x_{ij})]\|_{M_n(Y)}$, $\forall \ [x_{ij}] \in M_n(X)$ and
$\forall n \in \mathbb{N}$. Then, $\widetilde X$ and $Y$ are
completely isometrically isomorphic, and so $Min(X)$ and
$\widetilde X$ are completely isomorphic. In a similar way, we can
construct operator space structures on $X$ which are completely
isomorphic to $Max(X)$.

The following theorem describes the dual nature of minimal and
maximal operator space structures on a Banach space.
 \begin{theo}[\cite{b92}]
 For any Banach space $X$, we have $Min(X)^* \cong Max(X^*) $ and $ Max(X)^* \cong Min
 (X^*)$ completely isometrically.
 \label{minmaxduality}
 \end{theo}
 \section{Submaximal Spaces and $Q$-Spaces}
From the definition of minimal operator spaces, it is clear that
any subspace of a minimal operator space is again minimal, but a
quotient of a minimal space need not be minimal. An operator space
that is  a quotient of a minimal operator space (up to complete
isometric isomorphism) is called a {\emph{$Q$-space}} \cite{er}.
Since an operator space is minimal if and only if it is completely
isometric to a subspace of a commutative $C^*$-algebra
\cite{er02}, $Q$-spaces are precisely the quotients  of subspaces
of commutative $C^*$-algebras. Also, the category of $Q$-spaces is
stable under taking quotients and subspaces. $Q$-spaces were
investigated by M. Junge \cite{junge} and by Blecher and Le Merdy
\cite{blm}.

 $Q$-spaces need not be minimal, for instance, the
space $R \cap C$ is a $Q$-space, as it can be identified with the
quotient space $L^{\infty}[0,1]/S$, where $S$ is the subspace
orthogonal to the Rademacher functions \cite{gp98}. But, $R \cap
C$ is not minimal, and moreover $d_{cb}(R_n \cap C_n,
Min(\ell_2^n))= \sqrt{n}$ \ \cite{gp03}, so that $R \cap C$ is not
completely isomorphic to $Min(\ell_2)$.

Another example for a $Q$-space, which is not minimal is furnished
by the space of Hankel matrices that can be identified with
$L^{\infty}/H^{\infty}$. It can be shown that it has a subspace
which is  completely isometric to $R \cap C$, so that the space
$L^{\infty}/H^{\infty}$ is not minimal \cite{gp98}.

Subspace structure of various maximal operator spaces were studied
in \cite{tim04}. Subspaces of maximal operator spaces  are called
{\em submaximal\em} spaces  and in general they need not be
maximal, i.e., if $Y$ is a subspace of $X$ and if $x_{ij} \in Y$
for $i, j = 1, 2, ...,n$, then the norm of $[x_{ij}]$ in
$M_n(Max(Y))$ can be larger than the norm of $[x_{ij}]$ as an
element of $M_n(Max(X))$.  For example, the space $R + C$ is
submaximal, as it can be identified as a closed subspace of $Max (
L_1)$ spanned by the Rademacher functions\cite{hp93}.  But $R + C
$ is not maximal and moreover $d_{cb}(R_n + C_n, Max(\ell_2^n))=
\sqrt{n}$ \ \cite{gp03}, so that $R + C$ is not completely
isomorphic to $Max(\ell_2)$. However,  Paulsen  obtained the
following result.
\begin{theo}[\cite{vp96}]\label{vp}
Let $X$ be an infinite dimensional operator space and $x_{ij} \in
X$, for $ i, j = 1, 2, ...,n$, then \\ \centerline{$\left\|
[x_{ij}]
   \right\|_{M_n(Max(X))}= \mbox{inf} \ \{\  \left\| [x_{ij}]
   \right\|_{M_n(Max(Y))}\ |\  x_{ij} \in Y ,  Y \subset X, finite \ dimensional  \}
   $}
\end{theo}
 But
quotient spaces inherits the maximality as illustrated in the
following theorem \cite{gp03}.
\begin{theo}[\cite{gp03}]
\label{quot} If $X$ is a maximal operator space and $Y$ is a
closed subspace of $X$, then $Max(X/Y) \cong Max(X)/Y$ completely
isometrically.
\end{theo}
 Also, if every
subspace of $Max(X)$ is maximal, then any two Banach isomorphic
subspaces of $X$ will be completely isomorphic as subspaces of
$Max(X)$. For, if $E$ and $F$ are Banach isomorphic subspaces of
$X$, then  $E$ and $F$ are completely isomorphic as subspaces of
$Max(X)$. In \cite{vk1}, the notion of hereditarily maximal spaces
is introduced. Hereditarily maximal spaces determine a subclass of
maximal operator spaces with the property that the operator space
structure induced on any subspace coincides with the maximal
operator space structure on that subspace. Also, it is proved that
the class of hereditarily maximal spaces includes all Hilbertian
maximal operator spaces. Since $\ell_1^2 $ has a unique operator
space structure, $\ell_1^2$ is a maximal operator space and all of
its subspaces are maximal. So, $\ell_1^2$ is an example for a
hereditarily maximal space which is not Hilbertian. The smallest
submaximal space structure $\mu(X)$, admissible on an operator
space $X$ is studied in \cite{tim04} and \cite{vk2}.

The following result   reveals the natural duality between
subspaces and quotient spaces.
\begin{theo}[\cite{er02}] \label{du} If $ Y$ is a closed subspace of an operator space $X$,
then,\\ $(X/Y)^* \cong Y^{\bot}$ and $ Y^* \cong X^*/ Y^{\bot}$
completely isometrically, where \\ \centerline{$Y^{\bot} = \{ f
\in X^* \ | \ f(y)=0, \ \forall y \in Y \}$.}
\end{theo}

 Let $Y$ be a submaximal space, say $Y \subset Max(X)$.
Then by Theorem \ref{du}, $Y^* \cong (Max(X))^*/Y^{\bot}$. But by
using Theorem \ref{minmaxduality}, we get {$(Max(X))^*/Y^{\bot}
\cong Min(X^*)/Y^{\bot} $} showing that the dual of a submaximal
space is a $Q$-space.

\noindent Conversely, if  $Z= Min(X) /Y$ is a $Q$-space, then by
Theorem \ref{du}, $Z^*=(Min(X) /Y)^* \cong
 Y^{\bot}$.
 But here, \begin{eqnarray*} Y^{\bot} &=& \{ f \in (Min(X))^* \ | \
f(y)=0, \ \forall y \in Y \}\\
&=& \{ f \in Max(X^*) \ | \ f(y)=0, \ \forall y \in Y \}
\end{eqnarray*} so that, $Z^*$ is a submaximal space. Thus, the dual
of a submaximal space is a $Q$-space and vice versa.

We have noted that $Q$-spaces need not be
 minimal. But, using Theorem \ref{du},   we observe that
  if $X$ is a Hilbertian operator space, then any $Q$-space in $X$ (i.e., any quotient of $Min(X)$)  is minimal.

Let $X$ be a Hilbertian operator space and $Y$ be a closed
subspace of $X$.
 Let $id: Min(Y^{\bot}) \to Min(X)$ be the formal
identity mapping, and

 $\pi: Min(X) \to Min(X)/Y$ be the quotient
map.

Then $\pi \circ id : Min(Y^{\bot}) \to Min(X)/Y$ is a complete
contraction. Since, $X/Y$ is isometric to $Y^{\bot}$, $Min(X/Y)$
is completely isometric to $Min(Y^{\bot})$, so that
\\ \centerline{$\|[x_{ij}+Y]\|_{M_n(Min(X)/Y)} \leq
\|[x_{ij}+Y]\|_{M_n(Min(X/Y))}.$} But, by definition of minimal
operator spaces, \\ \centerline{$
\|[x_{ij}+Y]\|_{M_n(Min(X/Y))}\leq
\|[x_{ij}+Y]\|_{M_n(Min(X)/Y)}.$} Thus, $Min(X)/Y$ is completely
  isometrically isomorphic to $Min(X/Y)$.

\begin{rem}\hfill

There are non-Hilbertian operator spaces $X$ for which all
$Q$-spaces in $X$ are minimal. For instance, $\ell_1^2$ is a
minimal operator space, which is not Hilbertian and all $Q$-spaces
in $\ell_1^2$ are minimal. Identification of a subclass of minimal
operator spaces with the property that the operator space
structure induced on any quotient space coincides with the minimal
operator space structure on that quotient space is still open.
\end{rem}
\section{Main Results}
We have noted that a subspace of a minimal operator space is
minimal, whereas a subspace of a maximal space need not be
maximal.  On the other hand, let us consider the following
question: If every proper, nontrivial subspace of an operator
space $X$ is minimal (maximal), is $X$ minimal? (maximal?) The
answer to this question is No. For instance, if $X$ is of
dimension 2, and if $\widetilde{X}$ is any operator space
structure on $X$ such that $\widetilde{X} \neq Min(X)$ ( resp.
$\widetilde{X} \neq Max(X)$) (such a space exists since any Banach
space of dimension greater than 2 has more than one quantization,
i.e., the space has more than one admissible operator space
structure \cite{gp03}.) Then any proper, nontrivial subspace of
$\widetilde{X}$ will be of dimension 1, and so is minimal (resp.
maximal). (Since there is only one operator space of dimension 1,
up to complete isometric isomorphism.)

But in the case of infinite dimensional operator spaces, we have
an affirmative  answer.

\begin{theo} \label{subspace}Let $X$ be an infinite dimensional operator space.
 If every finite dimensional subspace of $X$ is minimal\ (maximal),
then $X$ is minimal\ (maximal).
\end{theo}
\begin{proof}\hfill

First, assume that every finite dimensional subspace of $X$ is
minimal. Let $[x_{ij}] \in M_n(X)$. Then, $[x_{ij}] \in M_n(E)$,
where  $E= \mbox{span} \ \{x_{ij}\}$ and the dimension of $E\leq
n^2$. By assumption, $E$ is minimal and by using the Hahn-Banach
theorem, we have
\begin{align*}
  \left\| [x_{ij}]
   \right\|_{M_n(X)}&=   \left\| [x_{ij}]
   \right\|_{M_n(E)}  \\
   &=   \left\| [x_{ij}]
   \right\|_{M_n(Min(E))}  \\
   &= \mbox{sup} \{ \left\| [f(x_{ij})]
   \right\| \ |\  \ f \in Ball(E^*) \}\\
& \leq \mbox{sup} \{ \left\| [f(x_{ij})]
   \right\| \ |\  \ f \in Ball(X^*) \}\\
   &= \left\| [x_{ij}]
   \right\|_{M_n(Min(X))}
\end{align*}
Since $Min(X)$ is the smallest operator space structure on $X$,
this shows that $\left\| [x_{ij}]
   \right\|_{M_n(X)} = \left\| [x_{ij}]
   \right\|_{M_n(Min(X))} $ and hence $X$ is minimal.

 Now to prove the maximal case, by Theorem~\ref{vp}, for any  $[x_{ij}] \in
M_n(X)$, we have
\begin{align*}
 \left\| [x_{ij}]
   \right\|_{M_n (Max(X))}&=  \mbox{inf} \{ \left\| [x_{ij}]
   \right\|_{M_n(Max(Y))} \ |\  x_{ij} \in Y, \  \ Y \subset X ,\ finite \ dimensional \}
   \\    &=  \mbox{inf} \{ \left\| [x_{ij}]
   \right\|_{M_n (Y)} \ | \ x_{ij} \in Y,  \  \ Y \subset X ,\ finite \ dimensional \}  \\
   &\leq    \left\| [x_{ij}]
   \right\|_{M_n(X)}
\end{align*}
Since $Max(X)$ is the largest operator space structure on $X$,
this shows that $\left\| [x_{ij}] \right\|_{M_n (Max(X))}= \left\|
[x_{ij}]\right\|_{M_n (X)}$ and hence $X$ is maximal.
\end{proof}

We have noted that  quotients of  minimal operator spaces need not
be minimal, whereas  quotients of  maximal operator spaces  are
maximal. In the case of quotient spaces of an infinite dimensional
operator space, we have the following result.
\begin{theo}\label{qq}
Let $X$ be an infinite dimensional operator space.\\ (i). If every
quotient of
 $X$ by a proper closed nontrivial subspace of $X$ is minimal, then
$X$ is minimal.\\
(ii). If every quotient of  $X$ by a proper closed nontrivial
subspace of $X$ is maximal, then $X$ is maximal.
\end{theo}
For proving this, we make use of the following theorems.
\begin{theo}[\cite{er02}]\label{norm}
If $X$ is any operator space and $[x_{ij}]\in M_n(X)$, there
exists a complete contraction $\varphi: X \rightarrow M_n  $ such
that $\| \varphi^{(n)}([x_{ij}])\|= \|[x_{ij}]\|.$
\end{theo}
\begin{theo}[\cite{bm04}]\label{quotientmap}
Let $X$ and $Z$ be operator spaces. If $\phi: X \to Z $ is
completely bounded, and if $Y$ is a closed subspace of $X$
contained in $ker(\phi)$, then the canonical map $\tilde{\varphi}:
X/Y \to Z$ induced by $\varphi$ is also completely bounded, with
$\| \tilde{\varphi}\|_{cb}= \|\varphi\|_{cb}.$ If $Y =
ker(\varphi)$, then $\varphi$ is a complete quotient map if and
only if $\tilde{\varphi}$ is a completely isometric isomorphism.
\vspace{0.0 in}
\end{theo}
\begin{proof} [Proof of Theorem \ref{qq}]\hfill

We first prove the second part. Let $[x_{ij}]\in M_n(X)$. By
Theorem~\ref{norm}, we have
\begin{eqnarray}\label{eq1}
  \|[x_{ij}]\|_{M_n(X)}&=& \mbox{sup} \{\|[\varphi(x_{ij})]\| \ \ | \ \varphi: X
\rightarrow M_n, \| \varphi \|_{cb}\leq1 \}
\end{eqnarray}
Since $X$ is infinite dimensional, and  the range of $\varphi$ is
finite dimensional,  $\varphi$ has a nontrivial kernel,
$ker(\varphi)$. Let $\widetilde{\varphi}: X/ker(\varphi) \to M_n$
be the canonical map defined by
$\widetilde{\varphi}(x+ker(\varphi))= \varphi(x)$. Then by
Theorem~\ref{quotientmap},
$\|\varphi\|_{cb}=\|\widetilde{\varphi}\|_{cb}=\|\widetilde{\varphi}\|=\|\varphi\|$,
where the last but one equality follows from the assumption that
$X/ker(\varphi)$ is maximal. Thus in equation (\ref{eq1}), we can
replace $\|\varphi\|_{cb}\leq1$ by $\|\varphi\|\leq1$, so that we
have
\begin{eqnarray}\label{eq2}
  \|[x_{ij}]\|_{M_n(X)}&=& \mbox{sup} \{\|[\varphi(x_{ij})]\| \ | \  \varphi: X
\rightarrow M_n, \| \varphi \|\leq1 \}
\end{eqnarray}
Applying equation (\ref{eq1}) to the operator space $Max(X)$, and
using the universal property of $Max(X)$ and the equation
(\ref{eq2}), we obtain
\begin{eqnarray*}
  \|[x_{ij}]\|_{M_n(Max(X))}&=& \mbox{sup} \{\|[\varphi(x_{ij})]\| \ | \ \varphi:
  Max(X)
\rightarrow M_n, \| \varphi \|_{cb}\leq1 \}\\ \end{eqnarray*}
\begin{eqnarray*}
 &=&\mbox{sup} \{\|[\varphi(x_{ij})]\| \ |\ \varphi:
  Max(X)
\rightarrow M_n, \| \varphi \|\leq1 \}\\
 &=&\mbox{sup} \{\|[\varphi(x_{ij})]\|\ |\  \varphi:
  X \rightarrow M_n, \| \varphi \|\leq1 \}\\
 &=&  \|[x_{ij}]\|_{M_n(X)}
\end{eqnarray*}
This shows that $X = Max(X)$ and so $X$ is maximal.

For proving the first part, we note that for any $[x_{ij}]\in
M_n(X)$,
\begin{eqnarray}\label{eq3}
\|[x_{ij}]\|_{M_n(X)}&=& \mbox{sup} \{\|[x_{ij}+Y]\|_{M_n(X/Y)} \
| \   Y  \subset  X \}
\end{eqnarray}
Now, we assume that every quotient of
 $X$ by a proper closed nontrivial subspace of $X$ is minimal.\\ Let $[x_{ij}]\in
M_n(X)$.  We
 claim that
\begin{eqnarray}  \|[x_{ij}]\|_{M_n(X)}&=&
\mbox{sup} \{\|[x_{ij}+Y]\|_{M_n(X/Y)}\ | \  Y  \subset  X ,
dim(X/Y)< \infty \} \nonumber
\end{eqnarray}
From equation (\ref{eq3}), we have
\begin{eqnarray}  \|[x_{ij}]\|_{M_n(X)}&\geq&
\mbox{sup} \{\|[x_{ij}+Y]\|_{M_n(X/Y)}\ |\ \  Y  \subset  X ,
dim(X/Y)< \infty \} \nonumber
\end{eqnarray}
By Theorem \ref{norm}, there exists a complete contraction
$\varphi: X \to M_n$ such that $\| \varphi^{(n)}([x_{ij}])\|=
\|[x_{ij}]\|_{M_n(X)}.$ Let $Y= ker(\varphi)$, then $X/Y$ is
finite dimensional and let $\tilde{\varphi}$ be the canonical map.
Then,$$\|[x_{ij}]\|_{M_n(X)} = \| \varphi^{(n)}([x_{ij}])\|= \|
\tilde{\varphi}^{(n)}([x_{ij}+Y])\| \leq
\|[x_{ij}+Y]\|_{M_n(X/Y)}.$$ This implies, \begin{eqnarray}
\|[x_{ij}]\|_{M_n(X)}&\leq& \mbox{sup}
\{\|[x_{ij}+Y]\|_{M_n(X/Y)}\ | \  Y  \subset  X, dim(X/Y)< \infty
\} \nonumber
\end{eqnarray} This proves the claim.

  Now let $id: Min(X) \to
X$ be the formal identity map. Let $Y$ be a subspace of $X$ such
that $ dim(X/Y)< \infty$ and $\pi: X \to X/Y$ be the quotient
mapping. Then, $\widetilde{id}= \pi \circ id : Min(X) \to X/Y$ is
given by $\widetilde{id}(x)= x+Y$. Since, by assumption, $X/Y$ is
minimal and  by using the universal property of minimal spaces,
$\widetilde{id}$ is completely bounded and
$\|\widetilde{id}\|_{cb}= \|\widetilde{id}\|\leq \|\pi\|\|id\|\leq
1$. This shows that $\widetilde{id}$ is completely contractive.
Therefore,
\begin{eqnarray*} \|[x_{ij}]\|_{M_n(X)} &=&\mbox{sup} \{\|[x_{ij}+Y]\|_{M_n(X/Y)} \ | \  Y \subset X,  dim(X/Y)<
\infty \}\\
&=& \mbox{sup} \{\|[\widetilde{id}(x_{ij})]\|_{M_n(X/Y)}\ | \ Y
\subset X,  dim(X/Y)< \infty \}\\
&\leq& \| [x_{ij}]\|_{M_n(Min(X))}.
\end{eqnarray*}
This shows that $X= Min(X)$, and so $X$ is minimal.
\end{proof}
\begin{rem}\hfill

If $X$ is finite dimensional, the above result need not be true.
For instance, if $X$ is an operator space of dimension 2 which is
not minimal (maximal), then every quotient of $X$ by a proper,
nontrivial closed subspace will be of dimension 1 and hence is
minimal (maximal).
\end{rem}

 \em Acknowledgements \em. The first
author acknowledge the financial support by University Grants
Commission of India, under Faculty Development Programme. The
first author is grateful to Prof. Gilles Pisier, Prof. Vern
Paulsen and Prof. Eric Ricard for having many discussions on this
subject.

Corresponding Author : Vinod Kumar. P

\end{document}